\documentclass[twopage]{article}

\newcommand{\titel}
{Fracpairs and fractions over a reduced commutative ring}
         
\usepackage{fancyhdr}
\usepackage{lastpage}
\usepackage{stmaryrd,amssymb,amsmath,amsthm,url}

\usepackage{environ}
\makeatletter
\NewEnviron{Lalign}{\tagsleft@true\begin{align}\BODY\end{align}}
\makeatother

\usepackage{environ}
\makeatletter
\NewEnviron{Ralign}{\tagsleft@false\begin{align}\BODY\end{align}}
\makeatother

\newtheorem{theorem}{Theorem}[subsection]

\newtheorem{proposition}[theorem]{Proposition}
\newtheorem{corollary}[theorem]{Corollary}
\newtheorem{definition}[theorem]{Definition} 
\theoremstyle{definition}

\newcommand{\lh}{\Bigl(}\newcommand{\rh}{\Bigl)}

\renewcommand{\bot}{\ensuremath{\textup{\textbf{a}}}}

\newcommand{\CR}{\ensuremath{\mathrm{CR}}}
\newcommand{\Md}{\ensuremath{\mathrm{Md}}}
\newcommand{\Mda}{\ensuremath{\mathrm{Md_\bot}}}

\newcommand{\Nat}{{\mathbb N}}
\newcommand{\NM}{\ensuremath{\mathbb M}}
\newcommand{\Int}{\ensuremath{\mathbb Z}}
\newcommand{\Rat}{\ensuremath{\mathbb Q}}

\newcommand{\I}{\ensuremath{\mathbb I}}

\newcommand{\cceq}{\ensuremath{\textup{cc}}}
\newcommand{\CCcm}{\ensuremath{\mathbb F_{cm}}}
\newcommand{\FPR}{\ensuremath{{\CCcm^{\,r}(R)}}}
\newcommand{\FPr}{\ensuremath{{\CCcm^{\,r}(\Int)}}}
\newcommand{\RFRS}{\ensuremath{\textup{\text{RF}}}}
\newcommand{\RFR}{\ensuremath{\textup{\textup{rf}}}}
\newcommand{\CC}{\ensuremath{\textup{\text{CC}}}}

\newcommand{\lfrac}[2]{#1 / #2}

\title{\titel}

\author{Jan A.\ Bergstra \& Alban\ Ponse
\\[2mm]
  {\small
	  Informatics Institute,
	  University of Amsterdam}\\
	{\small email: \url{j.a.bergstra@uva.nl}, \url{a.ponse@uva.nl}
	}
}
\date{}

\begin{document}

\maketitle

\thispagestyle{fancy}

\begin{abstract}
In the well-known construction of the field of fractions of an integral domain, division by zero
is excluded. 
We introduce ``fracpairs'' as pairs subject to laws consistent with the use 
of the pair as a fraction, but do not exclude denominators to be zero. 
We investigate fracpairs over a reduced commutative ring
(that is, a commutative ring that has no nonzero nilpotent elements) 
and provide these with natural definitions for addition, multiplication, and
additive and multiplicative inverse. We find that
modulo a simple congruence these fracpairs
constitute a ``common meadow'', which is a commutative monoid both for addition 
and multiplication,
extended with a weak additive inverse, a multiplicative inverse except for zero, 
and an additional element $\bot$ that is the image of the multiplicative inverse on
zero and that propagates through all operations. 
Considering $\bot$ as an error-value supports the intuition.

The equivalence classes of fracpairs thus obtained are called common cancellation
fractions (cc-fractions), and cc-fractions over the integers constitute a homomorphic 
pre-image of the common meadow $\Rat_\bot$, 
the field \Rat\ of rational numbers expanded with an \bot-totalized inverse.
Moreover, the initial common meadow is isomorphic to the initial algebra of cc-fractions
over the integer numbers.
Next, we define canonical term algebras (and therewith normal forms)
for cc-fractions over the integers and some meadows that model the
rational numbers expanded with a totalized inverse, 
and we provide some negative results concerning their associated
term rewriting properties.  
Then we consider reduced commutative rings in which the sum of two squares plus one cannot be a
zero divisor: by extending the equivalence relation on fracpairs we obtain an initial
algebra that is isomorphic to $\Rat_\bot$.
Finally, we express some negative conjectures
concerning alternative specifications for these (concrete) datatypes.
\\[2mm]
\end{abstract}

\section{Introduction}

In this paper we introduce \emph{fracpairs}, where the idea that 
``a fraction is a pair'' is formalized, though without the 
constraint that the second element of the pair must not be zero.
We provide fracpairs with natural definitions for addition, multiplication, and
additive and multiplicative inverse.
In order to further model this approach to a ``fraction'', one can consider fracpairs 
modulo any equivalence that is a congruence with respect to addition, multiplication, and
additive and multiplicative inverse, and we will consider two such equivalence relations.

\medskip

This set-up is comparable to the construction of the field of fractions of an integral 
domain,\footnote{Integral domain: a nonzero commutative ring in which the 
    product of any two nonzero elements is nonzero.}
which we recall here.
Given an an integral domain $R$, the elements of the field of fractions $Q(R)$ are equivalence 
classes in $R\times R\setminus\{0\}$ that are often represented as 
\[\frac pq\]
(in-line written as $p/q$), where the equivalence $\sim$ is defined by
\[\dfrac pq\sim\dfrac rs\quad\text{if, and only if}
\quad p\cdot s=q\cdot r\text{ holds in $R$}.\]
In $Q(R)$, addition, multiplication, and additive inverse are defined by
\begin{equation}
\label{eq:defs}
\dfrac p q+\dfrac r s=\dfrac{p\cdot s+r\cdot q}{q\cdot s}
\quad\text{and}\quad
\dfrac p q\cdot\dfrac r s=\dfrac{p\cdot r}{q\cdot s}
\quad\text{and}\quad
-\dfrac pq=\dfrac{-p}{q}
\end{equation}
and these definitions are independent from the particular choice of a representative $p/q$.
These fractions satisfy the axioms \CR\ given in Table~\ref{CR} of commutative rings with $0/p=0/1$
for the zero and $1/1$ for the multiplicative unit 1. 
Because each $p/q\in Q(R)$ different from
the zero has an inverse $q/p$, $Q(R)$ is a field, and it is the smallest field in which $R$ 
can be embedded. Identifying $p\in R$
with (the equivalence class of) $p/1$ makes $R$ a subring of $Q(R)$.

\begin{table}
\centering
\hrule
\begin{align*}
(x+y)+z &= x + (y + z)
& (x \cdot y) \cdot  z &= x \cdot  (y \cdot  z)\\
x+y &= y+x
& x \cdot  y &= y \cdot  x\\
x+0     &= x
& 1\cdot x &= x \\
x+(-x)  &= 0
& x \cdot  (y + z) &= (x \cdot  y) + (x \cdot  z)
\end{align*}
\hrule
\caption{\CR, axioms for commutative rings with $0$ as the zero and 1 as
the multiplicative unit}
\label{CR}
\end{table}

\medskip

In this paper we will consider fracpairs defined over a commutative ring $R$ that is 
\emph{reduced} 
(see \cite{Bourbaki}), i.e., $R$ has no nonzero nilpotent elements, or equivalently, $R$ 
satisfies the property
\begin{align}
\label{prop:rr}
x\cdot x=0~~\Rightarrow~~x=0.
\end{align}
The integral domain $\Int$ of integers is a prime example of a 
reduced commutative ring,\footnote{Terminology: Lam~\cite[p.194]{Lam} uses 
  ``commutative reduced ring'' and ``noncommutative reduced ring''.} 
and other examples that are not an integral domain
are the ring $\Int/6\Int$ and the ring $\Int\times\Int$. 

We recall the following  familiar consequences of the axioms \CR\ for commutative
rings: 
\[\text{$-0=0$, ~$0\cdot x=0$, ~$-(-x)=x$, and~
$-(x\cdot y)=x\cdot (-y)$.}\] 
As is common, we assume that $\cdot$ binds stronger than $+$ and
we will often omit brackets (as in $x \cdot  y + x \cdot  z$).

Fracpairs over a reduced commutative ring are provided with definitions for
addition, multiplication, and additive inverse as described in~\eqref{eq:defs}, and
| more interesting | also with a multiplicative
inverse. Our first main result (Thm.\ref{thm:1}) is that fracpairs
modulo a natural congruence relation constitute a so-called \emph{common meadow}. 
The equivalence classes of fracpairs
thus obtained will be called ``common cancellation fractions'', or \emph{cc-fractions} for short.
It follows that cc-fractions over \Int\ constitute 
a homomorphic pre-image of the common meadow $\Rat_\bot$, that is, the field $\Rat$ of rational numbers 
expanded with an $\bot$-totalized inverse (that is, $0^{-1}=\bot$).
A further result is the characterization of 
the initial common meadow as the initial algebra of cc-fractions over \Int\ (Thm.\ref{thm:2}).
Finally, for fracpairs over a reduced commutative ring
that satisfies a particular property we  consider a more identifying equivalence relation 
in order to define ``rational fractions'', and prove that the rational fractions over \Int\ 
represent $\Rat_\bot$ (Thm.\ref{thm:5}).
These results reinforce our idea that common meadows can be used in the development of alternative 
foundations of elementary (educational) mathematics from a perspective of abstract datatypes, 
term rewriting, and mathematical logic. We will return to this point in Section~\ref{sec:4}.

The paper is structured as follows. In Section~\ref{sec:2} we 
introduce fracpairs and cc-fractions over a reduced commutative ring and prove our main results. 
In Section~\ref{sec:3} we discuss some term rewriting issues for meadows 
in the context of fracpairs, and define canonical term algebras
that represent these meadows, including a representation of $\Rat_\bot$ as an 
initial algebra of rational fractions.
In Section~\ref{sec:4}, we end the paper with some conclusions and a brief digression.
In Appendix~\ref{app:0} we analyze the cc-fractions over $\Int/6\Int$, and in Appendix~\ref{app:1} we prove
some elementary identities for common meadows.

\section{Fracpairs and fractions over a reduced commutative ring}
\label{sec:2}
In Section~\ref{subsec:2.1} we define fracpairs and an equivalence on these, and establish some 
elementary properties.
In Section~\ref{subsec:2.2} we define common cancellation fractions and relate these to the setting of 
common meadows, and in Section~\ref{subsec:2.3} we present our main results. 

\subsection{Fracpairs and common cancellation equivalence: some elementary properties}
\label{subsec:2.1}
Given a reduced commutative ring $R$, a 
\emph{fracpair over $R$} is an element of $R\times R$ with special notation
\[\frac pq,\]
wich will be in-line written as $p/q$.
Note that for any $p\in R$, $p/0$ is a fracpair over $R$.
When considering a fracpair $p/q$ over $R$ as an expression, 
we will use some common terminology: 
\[
\text{$\dfrac pq$ has \emph{numerator $p$} and \emph{denominator $q$}.}
\]

We will consider fracpairs modulo some `cancellation equivalence', that is, 
an equivalence generated by a set of `cancellation identities', where a cancellation identity 
has the form $(x\cdot y)/(x\cdot z) = y/z$.

\begin{table}
\centering
\hrule
\begin{align}
\label{CC}
\tag{\CC}
	 \dfrac{x \cdot z }{y \cdot (z \cdot z) } &= \dfrac{x}{y \cdot z} 
\end{align}
\hrule
\caption{\CC, the  Common Cancellation axiom for fracpairs}
\label{tab:CC}
\end{table}

\begin{definition}
\label{def:cc-equivalence}
Let $R$ be a reduced commutative ring.
The cancellation equivalence
generated by the {common cancellation axiom \CC} defined in Table~\ref{tab:CC}
for fracpairs is called
\[\text{
\textbf{cc-equivalence}, notation $=_\cceq$.
}\]
\end{definition}

In the proposition below we state a few simple properties of cc-equivalence.

\begin{proposition}
\label{eq00}
For fracpairs over a reduced commutative ring $R$, the following identities hold:
\begin{align*}
\nonumber
\frac x{x\cdot x}&=_\cceq\frac 1x,
&
\frac{x\cdot z}{z\cdot z}&=_\cceq\frac xz,
\\
\dfrac x0&=_\cceq
\dfrac00,
&
\dfrac{x}{-y} &=_\cceq 
\dfrac{-x}{ y}.
\end{align*}
\end{proposition}

\begin{proof}
The two topmost identities are trivial: take $x=y=1$ respectively $y=1$ in \CC\ and apply
the axioms \CR\ for commutative rings.
Furthermore, by the top rightmost identity we immediately find
\[\dfrac x0=_\cceq\dfrac{x\cdot 0}{0\cdot0}=\dfrac00,\]
and
\[\dfrac{x}{-y} =_\cceq \dfrac{x \cdot (-y)}{(-y) \cdot (-y)} =
\dfrac{(-x) \cdot y}{y \cdot y} =_\cceq \dfrac{-x}{ y}.\]
\end{proof}

Consistency of the construction of fracpairs over $R$ amounts to the absence of 
unexpected identifications (thus, separations) in the case that $R$ is nontrivial ($0\ne 1$). 

\begin{proposition}
\label{prop:1}
Let $R$ be a nontrivial reduced commutative ring.
For fracpairs $p/0$ and $q/r$ over $R$ with nonzero $r$ it holds that
\[\dfrac p0\ne_\cceq \dfrac qr.\]
\end{proposition}

\begin{proof}
Each instance of \CC\ leaves the denominator $0$ of $\lfrac p0$ invariant:
if $s\cdot t=0$, then $s\cdot (t\cdot t)=0$ by \CR, 
and if $s\cdot (t\cdot t)=0$, then $(s\cdot t)\cdot(s\cdot t)=0$ by \CR\, and 
thus $s\cdot t=0$ by property \eqref{prop:rr} that defines reduced rings.
Hence, during a sequence of proof steps 
this denominator cannot transform from zero to nonzero or from nonzero to zero.
\end{proof}

In the remainder of this section we establish some more elementary properties of fracpairs, 
and discuss a related approach to ``fractions as pairs".
Let
\[n(R)=\{x\in R\mid \forall y\in R: ~x\cdot y =0~\Rightarrow~ y=0\}\]
be the set of non-zerodivisors of $R$.
It easily follows that 
\begin{equation}
\label{eqnn}
p\cdot q\in n(R)~\iff~ p\in n(R)\text{ and }q\in n(R).
~\footnote{Observe that for this property to hold, it suffices that $R$ is a commutative ring.}
\end{equation}

Let $R$ be a reduced commutative ring, then the relation $(p,q)\sim(r,s)$ 
defined by $p\cdot s= q\cdot r$ is an equivalence relation on $R\times n(R)$.
We show transitivity: assume $(p,q)\sim(r,s)\sim(u,v)$,
then $p\cdot s\cdot v=q\cdot r\cdot v=q\cdot s\cdot u$, hence $s\cdot (p\cdot v-q\cdot u)=0$. 
Since $s\in n(R)$, $p\cdot v-q\cdot u=0$, that is, $p\cdot v=q\cdot u$ and thus $(p,q)\sim(u,v)$.

\begin{proposition}
\label{prop:3}
Let $R$ be a reduced commutative ring and $p,q,r,s\in R$. 
If 
\[\text{$\dfrac{p}{q}=_\cceq\dfrac{r}{s}~$ and $~q\in n(R)$},\]
then $s\in n(R)$ and $R\models p\cdot s = q\cdot r$.
\end{proposition}

\begin{proof}
This follows by induction on the length of a proof of $p/q=_\cceq r/s$.
It suffices to show that each \CC-instance 
\[\frac{u\cdot w}{v\cdot (w\cdot w)}=_\cceq\frac u{v\cdot w}\]
implies 
\[v\cdot (w\cdot w)\in n(R)~\iff~v\cdot w\in n(R),\]
which follows from~\eqref{eqnn}, and that each such instance
satisfies $\sim$, which follows from \CR\  because
$(u\cdot w)\cdot (v\cdot w)=(v\cdot (w\cdot w))\cdot u$.
Hence, $s\in n(R)$ and thus $(p,q)\sim(r,s)$, and thus $R\models p\cdot s = q\cdot r$.
\end{proof}

Propositions~\ref{prop:1} and \ref{prop:3} imply the following corollary.

\begin{corollary}
\label{cor:CM}
Let $R$ be a nontrivial reduced commutative ring, then 
\begin{enumerate}
\item The fracpairs $\lfrac{0}{1}$, $\lfrac{1}{1}$, and $\lfrac{1}{0}$ over $R$
are pairwise distinct,
\item
For $p,q\in R$, if $\lfrac{p}{1}=_\cceq\lfrac{q}{1}$, then $R \models p = q$,
\item
For $p,q\in n(R)$, if $\lfrac{1}{p}=_\cceq\lfrac{1}{q}$, then $R \models p = q$.
\end{enumerate}
\end{corollary}

We conclude this section by discussing a construction that 
generalizes the notion of the field of fractions in a related way.
Let $R$ be an arbitrary commutative ring, and let $S$ be a multiplicative subset of $R$
(that is, $1\in S$ and if $u,v\in S$, then $u\cdot v\in S$). 
Then $S^{-1}R$, the \emph{localization of $R$ with respect to $S$} (see, e.g., \cite{Matsumura}), 
is defined as the set of equivalence classes of pairs $(p,q)\in R\times S$ under the equivalence relation
\[(p,q)\sim(r,s)~\iff~\exists u\in S:u\cdot(p\cdot s-q\cdot r)=0.\]
Addition and multiplication in $S^{-1}R$ are defined as usual (cf.\ the definitions in~\eqref{eq:defs}):
\[(p,q) + (r,s) = (p\cdot s+q\cdot r,q\cdot s)
\quad\text{and}\quad
(p,q) \cdot (r,s) = (p\cdot r, q\cdot s).
\]
For $S=n(R)$ this yields the \emph{total quotient ring of $R$},
also called the \emph{total ring of fractions of $R$} (see~\cite{Matsumura}). 
If $R$ is a domain, then $S=R\setminus\{0\}$ and the total 
quotient ring is the same as the field of fractions $Q(R)$. 
Since $S$ in the construction contains no zero divisors, the natural ring homomorphism 
$R \to Q(R)$ is injective, so the total quotient ring is an extension of $R$. 
In the general case, the ring homomorphism from $R$ to $S^{-1}R$ might fail to be injective. 
For example, if $0 \in S$, then $S^{-1}R$ is the trivial ring. In the case
that $S$ is also \emph{saturated}, that is, $x\cdot y\in S~\Rightarrow~x\in S$, we have the following
connection with fracpairs over $R$.

\begin{proposition}
\label{prop:4}
Let $R$ be a reduced commutative ring, and let $S$ be a multiplicative subset of $R$
that is saturated.

If $p/q=_\cceq r/s$ for $p,r,s\in R$ and $q\in S$, then $s\in S$ and $S^{-1}R\models (p,q)\sim(r,s)$.
\end{proposition}

\begin{proof}
Because $S$ is multiplicatively closed and saturated, $s\in S$ follows from $q\in S$
by induction on the length of a proof for $p/q=_\cceq r/s$. 
By Proposition~\ref{prop:3} it follows that $R\models p\cdot s=q\cdot r$, and
because $1\in S$, $S^{-1}R\models (p,q)\sim(r,s)$.
\end{proof}

In Appendix~\ref{app:0} we show how Proposition~\ref{prop:4} can be used to prove separation 
of fracpairs with respect to cc-equivalence: we show that
for certain fracpairs $p/q$ and $r/s$ over $\Int/6\Int$, $p/q\ne_\cceq r/s$.

\subsection{Fracpairs and common cancellation fractions: constants and operations}
\label{subsec:2.2}
In Table~\ref{tab:3} we define constants and operations for fracpairs that are
tailored to the setting of common meadows~\cite{BP14}, that is, structures over
the signature
\[
\Sigma_{cm}=\{0,1,\bot,-(\_),(\_)^{-1},+,\cdot\}.
\]
In the next section we explain the concept of a common meadow and discuss the role of 
the constants $0, 1$, and $\bot$.
Note that the defining equations for addition \eqref{F1}, multiplication~\eqref{F2}, and 
additive inverse~\eqref{F3} all have a familiar form.
The defining equation for the multiplicative inverse~\eqref{FP7} ensures that if a 
denominator of a fracpair has a factor $0$, then that of its inverse also has a factor $0$.
We shall sometimes omit brackets in sums and 
products of fracpairs and write 
\[\dfrac pq+\dfrac rs \quad\text{and}\quad \dfrac pq\cdot\dfrac rs.\]

\begin{table}
\hrule
\begin{minipage}[t]{0.5\linewidth}\centering
\begin{Ralign}
\label{F1}\tag{F1}
\lh\dfrac xy\rh + \lh\dfrac uv\rh
&= \dfrac{(x\cdot v) + (u \cdot y)}{y \cdot v}
\\[5mm]
\label{F2}\tag{F2}
\lh\dfrac xy\rh \cdot \lh\dfrac uv\rh&= \dfrac{x\cdot u}{y \cdot v}
\\[5mm]
\label{F3}\tag{F3}
-\lh\dfrac{x}{y}\rh &= \dfrac{-x}{y}
\\[5mm]
\label{FP7}\tag{F4}
\lh\dfrac{x}{y}\rh^{-1} & = \dfrac{y \cdot y}{x \cdot y}
\end{Ralign}
\end{minipage}
\hfill
\begin{minipage}[t]{0.40\linewidth}\centering
\begin{Ralign}
\label{F5}\tag{F5}
0 &= \dfrac{0}{1}\\[2mm]
\label{F6}\tag{F6}
1 &= \dfrac{1}{1}\\[2mm]
\label{F7}\tag{F7}
\bot &= \dfrac{1}{0}
\end{Ralign}
\end{minipage}\vspace{4mm}
\hfill

\hrule
\caption{Defining equations for the operations and constants of $\Sigma_{cm}$ on fracpairs}
\label{tab:3}
\end{table}

Given a reduced commutative ring $R$ we define 
\[F(R)\]
as the set of  fracpairs over $R$.

\begin{proposition} 
\label{prop:5}
Let $R$ be a reduced commutative ring, and let the meadow operations from $\Sigma_{cm}$ 
be defined on $F(R)$ by equations~$\eqref{F2}-\eqref{FP7}$ in Table~\ref{tab:3}.
Then the relation $=_\cceq$ is a congruence on $F(R)$.
\end{proposition}

\begin{proof}
It suffices to show that if 
$\lfrac{p}{q}$ can be proven equal to $\lfrac{r}{s}$ with finitely many 
instances of the axiom \CC, then the same holds for their
image under the meadow operations as defined in  
Table~\ref{tab:3}. 

Let 
$A = \lfrac{(p \cdot r)} {(q\cdot (r \cdot r))}$
and $B= \lfrac{p}{(q \cdot r)}$, so $A=_\cceq B$. Then
\begin{itemize}
\item
$A+(\lfrac st) =_\cceq B+(\lfrac st)$ because
\begin{align*}
\dfrac{p \cdot r} {q\cdot (r \cdot r)}+\dfrac st
&=\dfrac{(p \cdot r)\cdot t+s\cdot(q\cdot (r\cdot r))}{(q\cdot (r\cdot r))\cdot t}
&&\text{by \eqref{F1}}\\
&=\dfrac{(p \cdot t+s\cdot(q\cdot r))\cdot r}{(q\cdot t)\cdot (r \cdot r)}
=_\cceq\dfrac{p \cdot t+s\cdot(q\cdot r)}{(q\cdot t)\cdot r}\\
&=\dfrac{p \cdot t+s\cdot(q\cdot r)}{(q\cdot r)\cdot t}\\
&=\dfrac{p}{q\cdot r} + \dfrac st,
&&\text{by \eqref{F1}}
\end{align*} 
and $(\lfrac st)+ A =_\cceq (\lfrac st)+B$ follows in a similar
way,

\item
$A\cdot(\lfrac st) =_\cceq B\cdot(\lfrac st)$ follows immediately
from \eqref{F2},
and so does 
\[(\lfrac st)\cdot A=_\cceq (\lfrac st)\cdot B,\]

\item $-A=_\cceq -B$: trivial (by~\eqref{F3}),

\item
$\displaystyle A^{-1} =_\cceq B^{-1}$
because
\begin{align*}
\lh\dfrac{p \cdot r} {q\cdot (r\cdot r)}\rh^{-1} 
&=\dfrac{(q\cdot (r\cdot r))\cdot
(q\cdot (r\cdot r))}{(p\cdot r)\cdot(q\cdot (r\cdot r))}
&&\text{by \eqref{FP7}}\\
&=\dfrac{((q\cdot (r\cdot r))\cdot
(q\cdot r))\cdot r}{((p \cdot r)\cdot q)\cdot (r\cdot r)}
=_\cceq\dfrac{(q\cdot (r\cdot r))\cdot
(q\cdot r)}{((p\cdot r)\cdot q)\cdot r}\\
&=\dfrac{((q\cdot (r \cdot r))\cdot
q)\cdot r}{(p\cdot q)\cdot (r\cdot r)}
=_\cceq\dfrac{(q\cdot (r \cdot r))\cdot
q}{(p \cdot q)\cdot r}\\
&=\dfrac{(q\cdot r) \cdot (q\cdot r)}{p\cdot(q\cdot r)}\\
&=\lh\dfrac{p} {q \cdot r}\rh^{-1}.
&&\text{by \eqref{FP7}}
\end{align*}
\qed
\end{itemize}
\phantom\qedhere
\end{proof}

Equations~\eqref{F5},
~\eqref{F6}, and ~\eqref{F7} in Table~\ref{tab:3} define the 
constants $0, 1$, and $\bot$ from the common meadow signature $\Sigma_{cm}$
as fracpairs.
So, in the setting with fracpairs, these constants can be seen as abbreviations
for $0/1$, $1/1$, and $1/0$, respectively.

\begin{definition}
\label{def:init}
Let $R$ be a reduced commutative ring.
\begin{enumerate}
\item
A \textbf{common cancellation fraction over $R$}, \textbf{cc-fraction} for short, is a fracpair 
over $R$ modulo cc-equivalence. 
For a fracpair $p/q$ over $R$, $[p/q]_{\cceq}$ is the cc-fraction represented by $p/q$.
\item
The \textbf{initial algebra
of cc-fractions over $R$} equipped with the constants 
and operations of~Table~\ref{tab:3}, notation
\[\CCcm(R)\] 
is defined by dividing out cc-congruence on $F(R)$. Thus,
for fracpairs $p/q$ and $r/s$ over $R$,
\[\CCcm(R)\models \frac pq=\frac rs ~\iff~ 
(\CC + \text{\textup{Table~\ref{tab:3}}})\vdash \frac pq=_\cceq\frac rs
~\iff~[p/q]_{\cceq}=[r/s]_{\cceq}.\]
\end{enumerate}
\end{definition}

We will write $p/q$ for a cc-fraction $[p/q]_{\cceq}$ 
if it is clear from the context that a cc-fraction is meant.

\subsection{Common cancellation fractions constitute a common meadow}
\label{subsec:2.3}
With the aim of regarding the multiplicative inverse as a total
operation, \emph{meadows} were introduced by Bergstra and Tucker in~\cite{BT07} 
as alternatives for fields with a purely equational axiomatization.%
   \footnote{An overview of meadows as a new theme in the theory of rings and fields 
   is available at \url{https://meadowsite.wordpress.com/}.}
Meadows are commutative von Neumann regular rings~(vNrr's) equipped with a 
weak multiplicative inverse $x^{-1}$ (thus $0^{-1}=0$) that is an involution (thus $(x^{-1})^{-1}=x$).
In particular, the class of meadows is a variety, so each substructure of a meadow is a meadow,
which is not the case for commutative vNrr's (cf.~\cite{BB15}).
In this paper we will mainly consider a
variation of the concept of a meadow, and therefore meadows will be further referred to as
\emph{involutive meadows}. 

\medskip

A \emph{common meadow}~\cite{BP14} is a structure with addition, multiplication, and a multiplicative
inverse, and differs from an involutive meadow in that 
the inverse of zero is not zero, but equal to an additional 
constant $\bot$ that propagates through all operations.
Considering $\bot$ as an error-value supports the intuition.
Common meadows are formally defined as structures over the signature 
\(
\Sigma_{cm}=\{0,1,\bot,-(\_),(\_)^{-1},+,\cdot\}
\)
that satisfy the axioms in Table~\ref{Mda}, and we write \Mda\ for this set of axioms.
We further assume that 
the inverse operation $(\_)^{-1}$ binds stronger than $\cdot$ and omit
brackets whenever possible, e.g., $x\cdot (y^{-1})$ is written as $x\cdot y^{-1}$.

The use of the constant $\bot$
is a matter of convenience only, it constitutes a derived constant with defining equation
$\bot=0^{-1}$, so all uses of $\bot$ can be avoided.

\begin{table}[t]
\centering
\hrule
\begin{align*}
(x+y)+z &= x + (y + z)
 & -(-x)&=x\\
x+y &= y+x
 & 0 \cdot (x \cdot x) &= 0 \cdot x  \\
x+0 &= x
 & (x^{-1})^{-1} &= x + (0 \cdot (x^{-1}))\\
x+(-x) &= 0 \cdot x 
 & x \cdot (x^{-1}) &= 1  + (0 \cdot (x^{-1}))\\
(x \cdot y) \cdot z &= x \cdot  (y \cdot  z) 
 & (x \cdot y)^{-1} &= (x^{-1}) \cdot (y^{-1})\\
x \cdot  y &= y \cdot  x 
 & 1^{-1}&=1\\
1\cdot x &= x 
 & 0^{-1} &= \bot \\
x\cdot (y+z)&=(x\cdot y)+(x\cdot z)
 & x + \bot &= \bot\\
&&x \cdot \bot &= \bot
\end{align*}
\hrule
\caption{$\Md_\bot$, a set of axioms for common meadows}
\label{Mda}
\end{table}

\medskip

Before relating cc-fractions to common meadows, we provide some more introduction
to the latter. The axioms of \Mda\ that feature a (sub)term of the
form $0\cdot t$ cover the case that $t$ equals $\bot$, for example, $\bot+(-\bot)=0\cdot\bot=\bot$.
Some typical \Mda-consequences are these:
\[\text{$x=x+0\cdot x$, ~$0\cdot 0=0$,~$-0=0$,  ~and~
$-(x\cdot y)=x\cdot (-y)$}\] 
(we prove the last identity in Appendix~\ref{app:1}).
Another \Mda-consequence can be called the
\emph{weak additive inverse property}:
\[x+(-x)+x=x\]
(which follows with the axiom $x+(-x)=0\cdot x$), and thus by the axiom $-(-x)=x$ also $(-x)+x+(-x)=-x$.
We show that given $x$, any $y$ satisfying $x+y+x=x$ and $y+x+y=y$ is unique (implicitly using commutativity and
associativity):
\begin{align*}
y&=y+x+y\\
&=y+x+(-x)+x+y
\\
&=y+x+(-x)
&&\text{by $y+x+y=y$}\\
&=y+x+(-x)+x+(-x)
\\
&=(-x)+x+(-x)
&&\text{by $x+y+x=x$}\\
&=-x.
\end{align*}
Furthermore, by $-(x\cdot y)=x\cdot (-y)$ we find $-\bot=\bot$ , and 
with the axiom $(x^{-1})^{-1} = x + 0 \cdot x^{-1}$ we find $\bot^{-1}=\bot$.
In summary, a common meadow is 
a commutative monoid both for addition and multiplication,
extended with a weak additive inverse, a multiplicative inverse except for zero, and the additional
element $\bot$ that is the image of the multiplicative inverse on zero and propagates through all 
operations. 
Let $\NM_1$ and $\NM_2$ be common meadows, then
\[f:\NM_1\to\NM_2\]
is a homomorphism if $f$ preserves $1, 0$, and $\bot$, and the operations have the morphism
property (that is, $f(x+y)=f(x)+f(y)$, $f(x\cdot y)=f(x)\cdot f(y)$, $f(-x)=-f(x)$, and 
$f(x^{-1})=(f(x))^{-1}$). In the case that $\NM_1$ is a minimal algebra (that is,
each of its elements is represented by a closed term over $\Sigma_{cm}$), $f$ is unique.

\begin{proposition}
\label{prop:hom}
Let $\NM_1$ and $\NM_2$ be common meadows and $f:\NM_1\to\NM_2$ be a function that satisfies
\begin{align*}
f(x+y)&=f(x)+f(y)\\
f(x\cdot y)&=f(x)\cdot f(y)\\
f(x^{-1})&=(f(x))^{-1}\\
f(1)&=1,
\end{align*}
then $f$ is a homomorphism.
\end{proposition}
\begin{proof}
Write $1_i$ for the unit in $\NM_i$ and $0_i$ for its zero.
We first show $f(0_1)=0_2$: 
observe that $0_i=1_i+(-1_i)$, and in $\NM_2$, 
$1_2=f(1_1)=f(1_1+0_1)=f(1_1)+f(0_1)=1_2+f(0_1)$.
Hence
$0_2=1_2+(-1_2)=1_2+(-1_2)+f(0_1)=0_2+f(0_1)=f(0_1)$.
It follows that $f(\bot)=f(0_1^{-1})=0_2^{-1}=\bot$. Finally,
we have to prove that
$f(-x)=-f(x)$: observe $f(x)+(-f(x))=0_2\cdot f(x)=f(0_1)\cdot f(x)=f(0_1\cdot x)=f(x+(-x))=f(x)+f(-x)$,
and hence
\begin{align*}
f(-x)
&=f((-x)+x+(-x))
&&\text{by the weak additive inverse property}\\
&=f(-x)+f(x)+f(-x)\\
&=(-f(x))+f(x)+(-f(x))\\
&=-f(x).
&&\text{by the weak additive inverse property}
\end{align*}
\end{proof}
Given a common meadow \NM, we finally notice that after forgetting $(\_)^{-1}$, the 
substructure $\{x\in\NM\mid 0\cdot x=0\}$ is  a commutative ring.

\medskip

In the previous section we already suggested a strong connection between cc-fractions and 
common meadows if one forgets about the underlying ring $R$ and the fracpairing operation.
This yields the following elementary result, which together with the next corollary
we see as our first main result.

\begin{theorem}
\label{thm:1}
Let $R$ be a reduced commutative ring, then $\CCcm(R)$ is a common meadow.
\end{theorem}

\begin{proof}
By Proposition~\ref{prop:5}, 
$=_\cceq$ is a congruence with respect to $\Sigma_{cm}$.
Therefore, showing that $\CCcm(R)$ is a common meadow
only requires proof checking of all \Mda-axioms (see Table~\ref{Mda}).
We consider four cases, all other cases being equally straightforward:
\begin{align*}
\dfrac{p}{q} \cdot \lh\dfrac{p}{q}\rh^{-1} 
&= \dfrac{p}{q} \cdot \dfrac{q \cdot q}{p \cdot q} 
=\dfrac{p \cdot (q \cdot q)}{q \cdot (p \cdot q)}
= \dfrac{(p \cdot q) \cdot q }{p \cdot (q \cdot q)} \\
&= \dfrac{p \cdot q}{p \cdot q }
=\dfrac{1 \cdot (p \cdot q) + 0 \cdot 1}{ 1 \cdot (p \cdot q)}
=\dfrac{1}{1} + \dfrac{0}{p \cdot q}\\
&=1+ \dfrac{0}{1} \cdot \dfrac{q \cdot q}{p \cdot q}
=1 + 0 \cdot \dfrac{q \cdot q}{p \cdot q} 
= 1 + 0 \cdot \lh\dfrac{p }{ q}\rh^{-1},
\end{align*}
\begin{align*} 
\lh\lh\dfrac{p}{q}\rh^{-1}\rh^{-1} 
&=\lh\dfrac{q \cdot q}{p \cdot q}\rh^{-1}
=\dfrac{(p \cdot q)\cdot (p \cdot q)}{(p \cdot q) \cdot (q \cdot q)}
=\dfrac{((p \cdot q)\cdot p )\cdot q}{(p \cdot q) \cdot (q \cdot q)}\\
&=\dfrac{(p \cdot q)\cdot p}{(p \cdot q)\cdot q}
=\dfrac{p\cdot(p \cdot q)}{q\cdot(p \cdot q)}
=\dfrac{p}{q} + \dfrac{0\cdot q}{p \cdot q}\\
&=\dfrac{p}{q} + \dfrac01 \cdot \dfrac{q \cdot q}{p \cdot q}
=\dfrac{p}{q} + 0 \cdot \lh\dfrac{p}{q}\rh^{-1},
\end{align*}
\[
0^{-1}=\lh\dfrac 01\rh^{-1}=\dfrac{1\cdot 1}{0\cdot 1}=\dfrac10=\bot,
\quad\text{and }\quad
\dfrac pq+\bot=\dfrac pq+\dfrac 10=\dfrac q0\stackrel{\text{\ref{eq00}}}=\dfrac 10=\bot.
\]
\end{proof}

The construction of $\CCcm(R)$ is arguably the most straightforward construction of a 
common meadow.

With $\Rat_\bot$ we denote the common meadow that is defined as the field $\Rat$ of rational 
numbers expanded with an $\bot$-totalized inverse (that is, $0^{-1}=\bot$).%
  \footnote{$\Rat_\bot$ is introduced in \cite{BP14}.}
We have the following corollary of Theorem~\ref{thm:1}.

\begin{corollary}
\label{cor:1}
The unique homomorphism 
$f:\CCcm(\Int)\to\Rat_\bot$ 
is surjective, but not injective.
Thus, the common meadow $\CCcm(\Int)$
is a proper homomorphic pre-image of $\Rat_\bot$.
\end{corollary}

\begin{proof}
Observe that by Corollary~\ref{cor:CM}, $\CCcm(\Int)$ is nontrivial 
($0=\lfrac 01$, $1=\lfrac 11$, and $\bot=\lfrac01$ are pairwise distinct).
Define 
~$f:\CCcm(\Int)\to\Rat_\bot$~
by $f(\lfrac nm)=n\cdot m^{-1}$. 
Then $f$ is well-defined and according to Proposition~\ref{prop:hom} a homomorphism:
\begin{itemize}
\item
$f(x+y)=f(x) + f(y)$ because 
$\Rat_\bot\models x\cdot y^{-1}+u\cdot v^{-1}=(x\cdot v + u\cdot y)\cdot (y\cdot v)^{-1}$
(see Appendix~\ref{app:1} or \cite[Prop.2.2.2]{BP14}), 
\item
$f(x\cdot y)=f(x)\cdot f(y)$ follows trivially,
\item For the case $f((x)^{-1})$ first observe that $\Rat_\bot\models 0\cdot x\cdot x^{-1}=
0\cdot(1+0\cdot x^{-1})=0\cdot x^{-1}$, hence
$f((n/m)^{-1})=f((m\cdot m)/(n\cdot m))=m\cdot m\cdot(n\cdot m)^{-1}=m\cdot(1+0\cdot m^{-1})\cdot n^{-1}
=(m+0\cdot m^{-1})\cdot n^{-1}=(m^{-1})^{-1}\cdot n^{-1}=
(n\cdot m^{-1})^{-1}=(f(n/m))^{-1}$, and
\item $f(\lfrac 11)=1\cdot 1^{-1}=1$.
\end{itemize}

Each element in $\Rat_\bot$ can be represented by $n\cdot m^{-1}$ with $n,m\in\Int$,
hence $f$ is surjective. However, $f$ is not injective: $f(1/1)=f(2/2)=1$, while 
$\CCcm(\Int)\not\models 1/1=2/2$ because otherwise the homomorphism from 
$\CCcm(\Int)$ onto $\CCcm(\Int/6\Int)$
implies $\CCcm(\Int/6\Int)\models 1/1=2/2$, and the latter is a contradiction by Proposition~\ref{prop:4}, as is 
spelled out in Appendix~\ref{app:02}.
\end{proof}

Our second main result is a characterization of the initial common meadow.

\begin{theorem}
\label{thm:2}
The initial common meadow $\I(\Sigma_{cm},\Md_\bot)$ is isomorphic to $\CCcm(\Int)$.
\end{theorem} 

\begin{proof}
We use the following two properties of common meadows.
First, for each closed term $t$ over the meadow signature $\Sigma_{cm}$, there exist closed terms 
$p$ and $q$ over the signature $\Sigma_r=\{0,1,-(\_),+,\cdot\}$ of rings such that
$\Mda\vdash t=p\cdot q^{-1}$ (this follows by induction on the structure of $t$, applying 
the identity $\Mda\vdash x\cdot y^{-1}+u\cdot v^{-1}=(x\cdot v + u\cdot y)\cdot (y\cdot v)^{-1}
$). 
Secondly, $\Mda\vdash x\cdot (x^{-1}\cdot x^{-1})=x^{-1}$
(see Appendix~\ref{app:1} or \cite[Prop.2.2.1]{BP14}), and hence
\[\Mda\vdash (x\cdot z)\cdot (y\cdot (z\cdot z))^{-1}=x\cdot (y\cdot z)^{-1},\]
which can be seen as a characterization of~\CC\ (see Table~\ref{tab:CC}).

Because $\CCcm(\Int)$ is a model of \Mda, there exists a homomorphism
\[f: \I(\Sigma_{cm},\Md_\bot)\to \CCcm(\Int).\]
For $p$ a closed term over $\Sigma_r$, we find $f(p)=\lfrac p1$ (this follows
by structural induction on $p$), and thus
\[f((p)^{-1})=\lh\dfrac p1\rh^{-1}~\stackrel{\eqref{FP7}}=~\dfrac{1\cdot 1}{1\cdot p}=\dfrac 1p. \]
Hence, for $p,q$ closed terms over $\Sigma_r$,
$f(p\cdot q^{-1})=\lfrac pq$. 

It follows immediately that $f$ is surjective. 
Also, $f$ is injective: if for closed terms $p,q$ over $\Sigma_r$,
$f(p\cdot q^{-1})=f(r\cdot s^{-1})$, thus
\[\CCcm(\Int)\models\dfrac pq=\dfrac rs,\]
then we can find a proof using \CC\ and the \CR-axioms.
For closed terms over $\Sigma_r$, $\Mda$ implies the \CR-identities\footnote{In
  particular, $p+(-p)=0$ (or equivalently, $0\cdot p=0$); this follows  
  by structural induction on $p$.}
and each \CC-instance in this proof 
can be mimicked in $\I(\Sigma_{cm},\Md_\bot)$
with an instance of the equation 
$(x\cdot z)\cdot (y\cdot (z\cdot z))^{-1}=x\cdot (y\cdot z)^{-1}$.
Hence, $\Mda\vdash p\cdot q^{-1}=r\cdot s^{-1}$, 
so $\I(\Sigma_{cm},\Md_\bot)\models p\cdot q^{-1}=r\cdot s^{-1}$.
\end{proof}

\section{Term rewriting for meadows}
\label{sec:3}
In Section~\ref{subsec:3.1} we provide details about canonical terms for involutive 
meadows, for common meadows, and for cc-fractions.
Until now we have not been  successful in resolving questions about the existence 
of specifications for meadows with nice term rewriting properties, and we
provide in Section~\ref{subsec:3.2} a survey of relevant negative results. 
In Section~\ref{subsec:3.3} we define ``rational fractions" by defining 
an initial algebra that is isomorphic to $\Rat_\bot$. 

\subsection{DDRSes and canonical terms}
\label{subsec:3.1}
A so-called DDRS (datatype defining rewrite system, see~\cite{BP14a}) 
is an equational specification
over some given signature that, interpreted as a rewrite system by
orienting the equations from left-to-right, is ground complete and thus defines
(unique) normal forms for closed terms.
Given some DDRS, its \emph{canonical term algebra} (CTA) is determined as the
algebra over that signature with the set of normal forms as its domain, and 
in the context of CTAs we prefer to speak of \emph{canonical terms} rather than 
normal forms.
An abstract datatype (ADT) may be understood as the isomorphism class of its 
instantiations which are in our case CTAs.

In Table~\ref{DDRSZ} we define a DDRS for the 
ADT $\Int$ over the signature $\Sigma_r=\{0,1,-(\_),+,\cdot\}$ of rings.
Observe that the symmetric variant of equation~\eqref{eq:9}, that is,
\((-x)+(y+1)=((-x)+y)+1\),
is an instance of equation~\eqref{eq:5}. 

\begin{table}
\hrule
\begin{minipage}[t]{0.5\linewidth}\centering
\begin{Ralign}
\label{eq:3}\tag{r1}
-0&=0
\\
\label{eq:4}\tag{r2}
-(-x) &= x
\\[4mm]
\label{eq:5}\tag{r3}
x+(y+z) &= (x + y) + z
\\
\label{eq:6}\tag{r4}
x+0&= x\\
\label{eq:7}\tag{r5}
1+(-1)  &= 0 \\
\label{eq:8}\tag{r6}
(x + 1)+(-1)  &= x \\
\label{eq:9}\tag{r7}
x+(-(y+1))&=(x+(-y))+(-1)
\end{Ralign}
\end{minipage}\vspace{4mm}
\hfill
\begin{minipage}[t]{0.47\linewidth}\centering
\begin{Ralign}
\label{eq:10}\tag{r8}
	0+x     &= x\\
\label{eq:11}\tag{r9}
	(-1)+1  &= 0 \\
\label{eq:12}\tag{r10}
	(-(x+1)) + 1 &= -x\\
\label{eq:13}\tag{r11}
	(-x) + (-y) &= -(x+y)\\[2mm]
\label{eq:14}\tag{r12}
	x \cdot 0 &= 0  \\
\label{eq:15}\tag{r13}
	x\cdot 1 &= x \\
\label{eq:16}\tag{r14}
	x\cdot (- y) &= (-x)\cdot y \\
\label{eq:17}\tag{r15}
	x\cdot(y + z) &= (x\cdot y) + (x\cdot z)
\end{Ralign}
\end{minipage}
\hrule
\caption{A DDRS for $\Int$}
\label{DDRSZ}
\end{table}

\begin{definition}
\label{def:int}
\textbf{Positive numerals for $\Int$} are defined inductively:
$1$ is a positive numeral, and $n+1$ is a positive numeral if $n$ is one.
\textbf{Negative numerals for $\Int$} have the form $-(n)$ with $n$ a positive 
numeral. A \textbf{numeral for $\Int$} is either a positive or a negative numeral, or $0$. 
\\[1mm]
\textbf{Canonical terms for $\Int$} are the numerals for $\Int$, and we write
\[\widehat\Int\]
for the canonical term algebra for integers with these 
canonical terms.
\end{definition}

Thus, $\widehat\Int$ constitutes a datatype that implements (realizes) the ADT \Int\
by the DDRS specified in Table~\ref{DDRSZ}.
Some other specifications of \Int\ in the ``language of rings''
are discussed in~\cite{BP85}, but these have negative
normal forms 
\[-1, ~(-1)+(-1), ~((-1)+(-1))+(-1),...\] 

Below we define three more types of {canonical terms} 
and their associated canonical term algebras.
The (involutive) meadow $\Rat_0$ is defined as the field $\Rat$ of rational numbers 
with a zero-totalized inverse (so $0^{-1}=0$ and $(\_)^{-1}$
is an involution; see, e.g.,~\cite{BT07,BM11a,BBP13}).

\begin{definition}
\label{def:2}
\textbf{Canonical terms for $\Rat_0$} are the canonical terms for $\Int$
(see Definition~\ref{def:int}) and
closed expressions of the form $n \cdot m^{-1}$ and $(-n )\cdot m^{-1}$ such that
\begin{itemize}\setlength\itemsep{-.2em}
\item[$\ast$] $n$ is a positive numeral, and 
\item[$\ast$] $m$ is a positive numeral larger than 1, and 
\item[$\ast$] $n$ and $m$ (viewed as natural numbers) are relatively prime.
\end{itemize}
With $\widehat{\Rat_0}$ 
we denote the canonical term algebra for the abstract 
datatype $\Rat_0$ with these canonical terms. 
\end{definition}

Thus $\widehat{\Rat_0}$ is a datatype that implements the 
ADT $\Rat_0$.

\begin{definition}
\label{def:3}
\textbf{Canonical terms for $\Rat_\bot$} are the canonical terms for $\Rat_0$
and the additional constant $\bot$.
\\[1mm]
With $\widehat{\Rat_\bot}$ we denote the canonical term algebra for the abstract 
datatype $\Rat_\bot$ with these canonical terms. 
\end{definition}

Thus $\widehat{\Rat_\bot}$ 
is a datatype that implements the ADT $\Rat_\bot$.

\begin{definition}
\label{def:4}
\textbf{Canonical terms for $\CCcm(\Int)$} are all fracpairs
$\lfrac{n}{m}$ with $n$ and $m$ canonical terms for $\Int$
(see Definition~\ref{def:int}) and $m$ not a negative numeral,
such that one of the following conditions is met, where we write $m_\Int$ for the 
integer denoted by $m$:
\begin{itemize}\setlength\itemsep{-.2em}
\item[$\ast$] $n= 0$, and $m_\Int$ is squarefree, or
\item[$\ast$] $m=0$ and $n=1$, or
\item[$\ast$] $m = 1$, or
\item[$\ast$] $m \neq 0$ and $n \neq 0$ and  $m\neq 1$ and for every prime $p$,
if $m_\Int$ is a multiple of $p \cdot p$ then $n_\Int$ is not a multiple of $p$.
\end{itemize}
$\widehat\CCcm(\widehat\Int)$ is the canonical term algebra with these 
canonical terms.
\end{definition}
So, $\widehat\CCcm(\widehat\Int)$ constitutes a datatype 
that implements the ADT $\CCcm(\Int)$.

\subsection{Nonexistence of DDRSes for $\widehat\CCcm(\widehat\Int)$,
for $\widehat{\Rat_0}$, and for $\widehat{\Rat_\bot}$}
\label{subsec:3.2}
In this section we prove some negative results concerning the existence of 
certain DDRSes.

\begin{theorem}
\label{thm:3}
There is no DDRS for $\widehat\CCcm(\widehat\Int)$. 
\end{theorem}

\begin{proof}
Suppose $E$ is a finite set of rewrite rules for the signature of 
$\widehat\CCcm(\widehat\Int)$ that constitutes a DDRS. 
Notice that if $m$ is a positive numeral with $m_\Int$ not squarefree, 
then the fracpair 
\[\dfrac{0}{m}\] 
is not a normal form. Assume that $m$ exceeds the length 
of all left-hand sides of equations in $E$ (for some suitable measure),
thus $\lfrac{0}{m}$ must match with a left-hand side of say equation $e \in E$ 
that has the form 
\[\dfrac{0}{x\underline{+k}} \quad\text{or}\quad
\dfrac{y}{x\underline{+k}}\]
where we assume the following notational convention, writing $\equiv$ for syntactic
equivalence:
\[\text{$x\underline{+0} \equiv x$, and for all natural numbers $n$,
$x\underline{+(n+1)} \equiv (x \underline{+n}) +1$.}\]

Now choose a canonical term $\ell$ with $\ell_\Int$ squarefree and larger than 
$m_\Int$. It follows that $e$ rewrites 
$\lfrac{0}{\ell}$ so that this term cannot be a normal form which contradicts 
the definition of canonical terms (Definition~\ref{def:4}). 
\end{proof}

This proof works just as well if a DDRS is allowed to make use of auxiliary 
operations. Moreover, very similar proofs work for $\widehat{\Rat_0}$ and 
$\widehat{\Rat_\bot}$, as we state in the next theorem.

\begin{theorem}
\label{thm:4}
There is no DDRS for $\widehat{\Rat_0}$ and for $\widehat{\Rat_\bot}$. 
\end{theorem} 

\begin{proof}
Suppose $E$ is a finite set of rewrite rules for $\widehat{\Rat_0}$
that constitutes a DDRS and consider a term
$\lfrac{(1+1)}{m}$ with $m_\Int$ a multiple of $2$ that exceeds the largest equation in $E$
(for some suitable measure). Because $\lfrac{(1+1)}{m}$ is 
not a canonical term it is rewritten by say equation $e \in E$. The left-hand side 
of $e$ must have the form $\lfrac{t}{(x\underline{+k})}$ for 
some $t$ and $k$ so that $t$ matches with $1+1$. 
From this condition it follows that $x$ is not a variable in $t$ and without
loss of generality we may assume that 
\[t \in\{1+1,y,1+y, y+1, y+y\}.\]
Now let $\ell$ be a $\Rat_0$ ($\Rat_\bot$) canonical term so that $\ell_\Int $ 
is odd and exceeds $m_\Int$. We find that $\lfrac{(1+1)}{\ell}$ is a canonical term
according to the definition thereof but at the same time it is not a normal form 
because it can be rewritten by means of $e$. Thus, such $E$ does not exist.

Finally, observe that the above reasoning also applies for the case of 
$\widehat{\Rat_\bot}$.
\end{proof}

The above proof also demonstrates that auxiliary functions won't help, not 
even auxiliary sorts will enable the construction of a DDRS for $\widehat{\Rat_0}$ 
or for $\widehat{\Rat_\bot}$.
We notice that without the constraint that the normal forms are given in advance 
(by way of a choice of canonical terms) 
the matter is different because according to~\cite{BT95}, a DDRS can be found 
with auxiliary functions for each computable datatype.

We return to the question of DDRSes for rational numbers in Section~\ref{sec:4}, where
we express some (negative) conjectures about their existence.

\subsection{An initial algebra of fractions for rational numbers}
\label{subsec:3.3}
In this section we introduce ``rational fractions'', that is, fractions
tailored to an initial specification of the rational numbers in 
the style of $\CCcm(\Int)$. 
We start off with the definition of a certain class of reduced commutative rings.

Given a commutative ring $R$, consider the following conditional property:
\begin{equation}
\label{eq:BMCR}
\forall x,y,z \in R: x\cdot(y^2 + z^2 + 1) = 0  ~\Rightarrow~  x = 0.
\end{equation}
We first show that not each commutative ring that satisfies condition~\eqref{eq:BMCR} is reduced.%
  \footnote{Of course, not every reduced commutative ring
   satisfies property~\eqref{eq:BMCR}, for example $\Int/6\Int$ does not.}
The commutative ring $\Int[X]/(X^2)$, i.e., the polynomial ring in one indeterminate $X$ modulo the 
ideal generated by $X^2$, has as its elements polynomials of the form
$nX + m$ with $n,m\in \Int$ (see e.g.~\cite{Matsumura}). This ring  is not reduced 
($X\cdot X=0$ and $X\ne 0$), but satisfies property~\eqref{eq:BMCR}: suppose
\[(nX + m)\cdot((pX + q)^2 + (rX + s)^2+1) = 0,\]
thus 
\begin{equation}
\label{eq:ZZ}
(n(q^2 + s^2 + 1)+2m(pq+rs))X + m(q^2 + s^2 +1) = 0.
\end{equation}
Hence, $m(q^2 + s^2+1) = 0$, thus $m = 0$, and hence we find for $X$'s coefficient in~\eqref{eq:ZZ}
that also $n(q^2 + s^2+1)=0$, thus $n = 0$, and therefore $nX + m=0$.

\medskip

\begin{table}
\centering
\hrule
\begin{align}
\label{RFRS}
\tag{\RFRS}
\dfrac{x\cdot(((z\cdot z) + (u\cdot u)) + 1)}{y\cdot(((z\cdot z) + (u\cdot u)) + 1)} = \dfrac xy
\end{align}
\hrule
\caption{\RFRS, the rational fracpair axiom}
\label{tab:RFP}
\end{table}

\begin{definition}
\label{def:rf}
Let $R$ be a reduced commutative ring that satisfies property~\eqref{eq:BMCR}.
The cancellation equivalence
generated by the {rational fracpair axiom \RFRS} defined in Table~\ref{tab:RFP} 
and the common
cancellation axiom \CC\ for fracpairs (defined in Table~\ref{tab:CC})
is called
\[\text{\textbf{rf-equivalence}, notation $=_\RFR$.}
\]
\end{definition}

Let $R$ be a nontrivial reduced commutative ring that satisfies property~\eqref{eq:BMCR}. 
Concerning consistency (and thus separation) it follows for fracpairs $p/q$ and $r/0$ over $R$ 
that if $q$ is nonzero, then  
\[p/q\ne_\RFR r/0\]
because~\eqref{eq:BMCR} 
ensures that application of \RFRS\ cannot turn a nonzero denominator into zero or vice versa 
(cf.\ Proposition~\ref{prop:1}).
Furthermore, as in Corollary~\ref{cor:CM}, it follows that 
\begin{itemize}
\item
for $p,q\in R$, if $\lfrac{p}{1}=_\RFR \lfrac{q}{1}$, 
then $R \models p = q$, and
\item for $p,q\in n(R)$, if $\lfrac{1}{p}=_\RFR \lfrac{1}{q}$, then $R \models p = q$.
\end{itemize}

With respect to the operations on fracpairs in Table~\ref{tab:3} (thus, with respect to the signature of
common meadows) it follows that {$=_\RFR$} is a congruence on $F(R)$, the set of fracpairs over $R$
(cf.\ Proposition~\ref{prop:5}). 
\begin{definition}
\label{def:rational}
Let $R$ be a reduced commutative ring that satisfies property~\eqref{eq:BMCR}.
\begin{enumerate}
\item
A \textbf{rational fraction over $R$} is a fracpair 
over $R$ modulo rf-equivalence. 
\item
The \textbf{initial algebra
of rational fractions over $R$} equipped with the constants 
and operations of~Table~\ref{tab:3}, notation
\[\FPR\]
is defined by dividing out rf-congruence on $F(R)$.\end{enumerate}
\end{definition}
We end this section with an elementary result for the particular case of $\FPr$. 

\begin{theorem}
\label{thm:5}
The structure
$\FPr$ is a common meadow that is isomorphic to $\Rat_\bot$.
\end{theorem} 

\begin{proof}
In~\cite{BM11a} the following folk theorem in 
field theory is recalled (and proven, see Lemma~7): For each prime number $p$
and $u\in \Int_p$, there exist $v, w \in\Int_p$ such that $u=v^2+w^2$. This implies
the following property (see~\cite[Corollary~1]{BM11a}\footnote{The report version of this paper
     (\texttt{arXiv:0907.0540v3}) uses a different numbering: Lemma~6 and Corollary 1, respectively.}):  
\begin{equation}
\label{eq:BM}
\text{For each prime number $p$ there exist $a,b,m\in\Nat$
such that $m\cdot p = a^2 + b^2+1$.}
~\footnote{A proof of~\eqref{eq:BM} is the following: let $a,b\in \Int_p$ be such that $-1 = a^2 +b^2$. 
Then $a^2 +b^2 +1$ is a multiple of $p \in \Nat$.}
\end{equation}
Now, given some prime number $p$, let $m,a,b$ be such that~\eqref{eq:BM} is satisfied.
For arbitrary $c,d\in\Nat$ we derive
\begin{align*}
\dfrac{ c\cdot p}{d\cdot p} 
&=_\RFR~\dfrac{c\cdot p\cdot (a^2 + b^2+1)}{d\cdot p\cdot (a^2 + b^2+1)}
&&\text{by \ref{RFRS}}\\[2mm]
&= \quad\dfrac{c\cdot p\cdot m\cdot p}{d\cdot p\cdot m\cdot p}
&&\text{by \eqref{eq:BM}}\\[2mm]
&=_\RFR~\dfrac{c\cdot m\cdot p}{d\cdot m\cdot p}
&&\text{by \CC}\\[2mm]
&= \quad\dfrac{c\cdot (a^2 + b^2+1)}{d\cdot (a^2 + b^2+1)}
&&\text{by \eqref{eq:BM}}\\[2mm]
&=_\RFR~\dfrac cd.
&&\text{by \ref{RFRS}}
\end{align*}
So, for $n,m\in\Nat$ it follows that:
\begin{enumerate}
\item[$\ast$]
$\lfrac nm =_\RFR~\lfrac pq$\quad with $p,q$  relative prime if $n\ne0\ne m$,
\item[$\ast$]
$\lfrac nm =_\RFR~\lfrac 01$\quad if $n = 0$ and $m\ne0$,
\item[$\ast$]
$\lfrac nm =_\RFR~\lfrac 10$\quad if $m = 0$ (cf.\ Proposition~\ref{eq00}).
\end{enumerate}
Hence, we can represent each rational fraction by a fracpair that matches 
the definition of canonical terms for $\Rat_\bot$, identifying
$\lfrac nm$ with $n\cdot m^{-1}$ if $n\ne 0$ and $m\not\in\{0,1\}$, 
with $n$ if $m=1$ or [$n=0$ and $m\ne0$], and with $\bot$ if $m=0$
(cf.~Definition~\ref{def:2}). 

The observation that the defining equations for the constants and operations
of common meadows given in Table~\ref{tab:3} match those for
$\Rat_\bot$ finishes the proof.
\end{proof}

\section{Conclusions and digression}
\label{sec:4}
We lifted the notion of a quotient field construction 
by dropping the requirement that in a ``fraction $\lfrac pq$'' (over
some integral domain) the $q$ must 
not be equal to zero and came up with the notion of \emph{fracpairs} defined
over a reduced commutative ring $R$, and \emph{common cancellation fractions} (cc-fractions)
that are defined by a simple equivalence on fracpairs over $R$.
Natural definitions of the constants and operations of a meadow on 
fracpairs yield a common meadow (Thm.\ref{thm:1}), and
this is arguably the most straightforward construction of a common meadow. 
Furthermore, we showed that the common meadow $\CCcm(\Int)$ of common cancellation fractions
over \Int\ is a proper homomorphic 
pre-image of $\Rat_\bot$ (Cor.\ref{cor:1}), and is isomorphic 
to the initial common meadow
(Thm.\ref{thm:2}; confer the characterization of the involutive
meadow in~\cite{BR10}).

Then, in Section~\ref{sec:3}, we considered canonical terms and
term rewriting for integers and for some meadows that model expanded versions
of the rational numbers,
and proved the nonexistence of DDRSes (datatype defining rewrite systems)
for the associated canonical term algebras of $\CCcm(\Int)$, $\Rat_0$ and 
$\Rat_\bot$ 
(Thm.\ref{thm:3} and Thm.\ref{thm:4}), each of which is based on a DDRS in which the
integers are represented as  $0$, the positive numerals 
\[1, ~1+1, ~(1+1)+1,...,\]
and the negations thereof.
Moreover, we defined ``rational fracpairs'' that constitute an initial algebra that 
is isomorphic to $\Rat_\bot$ (Thm.\ref{thm:5}).

\medskip

We have the following four conjectures
about the nonexistence of DDRS specifications for rational numbers:
\begin{enumerate}
\item \label{Cone}
The meadow of rationals $\Rat_0$ admits an equational  initial algebra 
specification (see~\cite{BT95} and a 
subsequent simplification in ~\cite{BM11a}). 
Now the conjecture is that no finite equational initial algebra specification 
of $\Rat_0$ is both confluent and strongly terminating (interpreting the equations
as left-to-right rewrite rules). 
This is irrespective of the choice of normal forms. 

Another formulation of this conjecture: $\Rat_0$
cannot be specified by means of a DDRS.

\item \label{Cthree}
We conjecture that for $\Rat_\bot$ the same situation applies as for $\Rat_0$: 
No DDRS for it can be found irrespective of the normal forms one intends the 
DDRS to have.

\item \label{Cfour}
The following conjecture (if true) seems to be simpler to prove: 
$\CCcm(\Int)$ cannot be specified by means of a DDRS.
 
\item The above negative conjectures remain if one allows the DDRS to be modulo 
associativity of $+$ and $\cdot$, commutativity of $+$ and $\cdot$, or both 
associativity and commutativity $+$ and $\cdot$.
\end{enumerate}
Concerning these matters, we should mention the work~\cite{Contejean} of 
Contejean~\emph{et al} in which
normal forms for rational numbers are specified by a complete term rewriting
system modulo
commutativity and associativity of $+$ and~$\cdot$. The associated datatype
$\texttt{Rat}$ comprises two functions
\[\texttt{rat},\slash:\Int\times\Int\to\texttt{Rat},\]
where the symbol $\texttt{rat}$ denotes any fraction, while
$\slash$ denotes irreducible fractions.
Also in this work, division by zero is allowed, ``but such alien terms can be
avoided by introducing a sort for non-null integers'' and is not considered 
any further. The main purpose of this work is to use the resulting datatype
for computing Gr\"obner bases of polynomial ideals over $\Rat$.

\medskip

We conclude with some comments on the use of the word ``fraction'', a term that is sometimes 
used in the semantic sense, 
as in the field of fractions, and sometimes in the syntactic sense, as a fraction having a numerator and a 
denominator. 
For the latter interpretation we introduced the notion of a ``fracpair'' to be used if 
numerator and denominator are viewed as values, and in the case
that we want to refer to the particular syntax of numerator and denominator, 
one can introduce the notion of a \emph{fracterm}, that is, an ``expression of type fracpair''
(thus, not making any identifications that hold in the underlying ring).
Rollnik~\cite{Rollnik} prefers to view fractions a values, over viewing fractions as pairs or 
viewing fractions as terms. He develops a detailed teaching method for fractions based on that viewpoint.
Fracpairs provide an abstraction level in between of both views of fractions.

Finally we comment on a classic requirement on addition of fractions: 
\begin{align}
\label{eq:20b}
\dfrac{x}{y} + \dfrac{z}{y}&= \dfrac{x+ z}{y}.
\end{align}
With the axiom \CC\ and the defining equation for + (see Table~\ref{tab:3})
a proof of this law is immediate:
\[\dfrac{x}{y} + \dfrac{z}{y}= \dfrac{(x\cdot y) + (z \cdot y)}{y \cdot y} =  
\dfrac{(x+z) \cdot y}{y \cdot y} =_\cceq \dfrac{x+z}{y}.\]
Taking $\Int$ as the underlying reduced commutative ring, this relates to the notion 
of \emph{quasi-cardinality} that emerged from educational mathematics and 
is due to Griesel~\cite{Griesel} (see also Padberg~\cite[p.30]{Padberg}). 
The aspect of quasi-cardinality for 
addition of fracpairs, which can also be called the
\emph{quasi-cardinality law}, is expressed by equation~\eqref{eq:20b}.
So we find that the quasi-cardinality law, which features as a central fact in many 
textbooks on elementary arithmetic, follows from the equations for fracpairs, the 
definition of addition on fracpairs, and the \CC-axiom.

\paragraph{Acknowledgement} We thank Kees Middelburg and Jan Willem Klop
for useful comments on an earlier version
of this paper.
In addition, we thank two referees for a number of important suggestions, in particular
concerning Sections~\ref{subsec:2.1} and~\ref{subsec:3.3}.

\appendix
\section{Fracpairs over $\Int/6\Int$ and the structure of $\CCcm(\Int/6\Int)$}
\label{app:0}
In Appendix~\ref{app:01} we discuss cc-equivalence of fracpairs over $\Int/6\Int$, and in
Appendix~\ref{app:02} we
analyze the structure of $\CCcm(\Int/6\Int)$. 

\subsection{Cc-equivalence of fracpairs over $\Int/6\Int$}
\label{app:01}
In this section we investigate which constants can be used to represent all fracpairs
over the reduced commutative ring $\Int/6\Int$ modulo cc-equivalence.
Recall that in $\Int/6\Int$,
\[\text{$-0=0$, $-1=5$, $-2=4$, and $-3=3$,}\] 
and that addition and multiplication are defined by 
\[
\def\arraystretch{1.2}
\begin{array}[t]{c|cccccc}
+&1&2&3&4&5
\\\hline
1
&2&3&4&5&0\\
2
&3&4&5&0&1\\
3
&4&5&0&1&2\\
4
&5&0&1&2&3\\
5
&0&1&2&3&4
\end{array}
\qquad\qquad
\begin{array}[t]{c|cccccc}
\cdot&1&2&3&4&5
\\\hline
1&1&2&3&4&5\\
2&2&4&0&2&4\\
3&3&0&3&0&3\\
4&4&2&0&4&2\\
5&5&4&3&2&1
\end{array}
\]

\renewcommand{\frac}{\dfrac}

From the constants in $\Int/6\Int$ one obtains 36 fracpairs, from which the following twelve 
can be used to represent all fracpairs modulo cc-equivalence:
\begin{equation}
\label{repr}
\begin{array}{l}
\frac10,\quad \frac 01,\frac11,\frac21,\frac31,\frac41,\frac51,\quad 
\frac02, \frac12, \frac22,\quad \frac03, \frac13.
\end{array}
\end{equation}
In Table~\ref{tab:6} we show that the fracpairs listed in~\eqref{repr} represent 
all fracpairs over $\Int/6\Int$ modulo cc-equivalence.
For all $p,q\in\Int/6\Int$ and $r\in\Int/6\Int\setminus\{0\}$, Proposition~\ref{prop:1} implies 
$p/r\ne_\cceq q/0$, and Proposition~\ref{eq00} implies $p/0=_\cceq q/0$. 
Of course, we choose $1/0$ as the representing fracpair for the latter equivalence.

\begin{table}
\centering
\hrule
$
\def\arraystretch{2.8}
\begin{array}{l|l}
\\[-8mm]
\quad
\begin{array}[t]{l}
\frac32=_\cceq \frac{3\cdot2}{2\cdot2}=_\cceq \frac {0\cdot2}{2\cdot 2\cdot2}= \frac02\\
\frac42=  \frac{2\cdot2}{2\cdot2\cdot2}=_\cceq \frac 2{2\cdot 2}=_\cceq \frac 12\\
\frac52=_\cceq \frac{5\cdot 2}{2\cdot 2}=_\cceq  \frac{5\cdot2\cdot2}{2\cdot 2\cdot2}=\frac22\\
\end{array}
\quad&\quad
\begin{array}[t]{l}
\frac 23=  \frac2{3\cdot3}=_\cceq \frac {2\cdot3}{3\cdot 3\cdot3}= \frac 03\\
\frac33=  \frac{3}{3\cdot3}=_\cceq \frac 13\\
\frac43=  \frac{4}{3\cdot3}=_\cceq \frac {4\cdot3}{3\cdot3\cdot3}= \frac 03\\
\frac53=  \frac{5}{3\cdot5\cdot5}=_\cceq \frac 1{3\cdot 5}=\frac13
\\[5mm]
\end{array}
\quad
\\
\hline
\\[-8mm]
\quad
\begin{array}[t]{l}
\frac04= \frac{0\cdot2}{2\cdot2}=_\cceq \frac{0}{2}\\
\frac14= \frac1{2\cdot2}=_\cceq\frac {1\cdot2}{2\cdot2\cdot2}=\frac 22\\
\frac24= \frac{2}{2\cdot2}=_\cceq\frac 12\\
\frac34= \frac{3}{2\cdot2}=_\cceq\frac {3\cdot2}{2\cdot 2\cdot2}=\frac 02\\
\frac44= \frac{2\cdot2}{2\cdot 2}=_\cceq\frac 22\\
\frac54= \frac{5}{4\cdot5\cdot5}=_\cceq\frac 1{4\cdot 5}=\frac 12
\end{array}
\quad&\quad
\begin{array}[t]{l}
\frac05= \frac{0\cdot5}{5\cdot5\cdot5}=_\cceq\frac 0{5\cdot 5}=\frac 01\\
\frac15= \frac{5\cdot5}{5\cdot5\cdot5}=_\cceq\frac 5{5\cdot 5}=\frac 51\\
\frac25= \frac{4\cdot5}{5\cdot5\cdot5}=_\cceq\frac 4{5\cdot 5}=\frac 41\\
\frac35= \frac{3\cdot5}{5\cdot5\cdot5}=_\cceq\frac 3{5\cdot 5}=\frac 31\\
\frac45= \frac{2\cdot5}{5\cdot5\cdot5}=_\cceq\frac 2{5\cdot 5}=\frac 21\\
\frac55= \frac{5}{5\cdot5\cdot5}=_\cceq\frac 1{5\cdot 5}=\frac 11
\\[5mm]
\end{array}
\\
\hline
\\[-8mm]
\quad
\begin{array}[t]{l}
\frac00=_\cceq\frac20=_\cceq\frac30=_\cceq\frac40=_\cceq\frac50=_\cceq\frac10
\end{array}
\\[5mm]
\end{array}
$\hrule
\caption{Equivalences between fracpairs over $\Int/6\Int$ modulo cc-equivalence,
where the righthand sides occur in~\eqref{repr} and where Proposition~\ref{eq00} is repeatedly used}
\label{tab:6}
\end{table}

\medskip

In the following we prove separation modulo cc-equivalence of various fracpairs over $\Int/6\Int$ 
using Proposition~\ref{prop:4}.
There are three choices for a saturated subset $S$ of $\Int/6\Int$ that yield a nontrivial localized ring
$S^{-1}(\Int/6\Int)$. We list the equivalences generated by each of these subsets:
\begin{description}
\item[$S=\{1,5\}:$]~
$(\textbf{\textit k},\mathbf 1)\sim(5k,5)\text{\quad for }k\in\Int/6\Int$
\item[$S=\{1,3,5\}:$]~
\arraycolsep=2pt
$\begin{array}[t]{ll}
(0,1)\sim(2,1)\sim(4,1)\sim\mathbf{(0,3)}\sim(2,3)\sim(4,3)\sim(0,5)\sim(2,5)\sim(4,5)
\\[1mm]
(1,1)\sim(3,1)\sim(5,1)\sim\mathbf{(1,3)}\sim(3,3)\sim(5,3)\sim(1,5)\sim(3,5)\sim(5,5)
\end{array}$
\item[$S=\{1,2,4,5\}:$]~
$\begin{array}[t]{ll}
&\mathbf{(0,1)}\sim\mathbf{(3,1)}\sim\mathbf{(0,2)}\sim(3,2)\sim(0,4)\sim(3,4)\sim(0,5)\sim(3,5)
\\[1mm]
&\mathbf{(1,1)}\sim\mathbf{(4,1)}\sim\mathbf{(2,2)}\sim(5,2)\sim(1,4)\sim(4,4)\sim(2,5)\sim(5,5)
\\[1mm]
&\mathbf{(2,1)}\sim\mathbf{(5,1)}\sim\mathbf{(1,2)}\sim(4,2)\sim(2,4)\sim(5,4)\sim(1,5)\sim(4,5)
\end{array}$
\end{description}
The proofs of these equivalences are trivial but cumbersome.
With respect to fracpairs over $\Int/6\Int$, the following can be concluded:
\begin{enumerate}
\item 
The case $S=\{1,5\}$ implies
that $k/1\ne_\cceq\ell/1$ if $k\ne\ell$, and that fracpairs of the 
form $x/5$ need not be considered. So this case yields six fracpairs that are distinct 
modulo cc-equivalence.
\item  
The case $S=\{1,3,5\}$ introduces six fracpairs of the form $x/3$,
and implies that $0/3$ and $1/3$ are distinct, that $0/3$ is distinct from
$1/1$, $3/1$ and $5/1$, and $1/3$ from $0/1$, $2/1$ and $4/1$.
Furthermore, the identities in Table~\ref{tab:6} imply
that $0/3$ and $1/3$ represent  modulo cc-equivalence all fracpairs of the form $x/3$.
\item
The case  $S=\{1,2,4,5\}$ introduces twelve fracpairs of the form $x/2$ or $x/4$,
and the identities in Table~\ref{tab:6} imply that $0/2$ and $1/2$ and $2/2$ represent all fracpairs 
of this form.
Furthermore, this case implies that $0/2$ and $1/2$ and $2/2$ are mutually distinct modulo cc-equivalence.
\end{enumerate}
We note that e.g.\ $1/1$ and $2/2$ can not be distinguished in this way. Separation of $1/1$ and $2/2$ 
in $\CCcm(\Int/6\Int)$
can however be proven easily, as we show in Appendix~\ref{app:02}, and separations of the remaining fracpairs 
from~\eqref{repr} that do not follow from the conclusions above can be proven in a similar fashion. 

\subsection{The structure of $\CCcm(\Int/6\Int)$}
\label{app:02}

By Theorem~\ref{thm:1}, $\CCcm(\Int/6\Int)$ is a common meadow.
Multiplication in $\CCcm(\Int/6\Int)$ is defined in Table~\ref{tab:mul}, 
where we leave out 
$1/0=\bot$ (recall $\Mda\vdash x\cdot \bot = \bot$). 

\begin{table}
\centering
$
\arraycolsep=11.8pt\def\arraystretch{2.6}
\begin{array}{c|cccccc|ccc|cc}
x\cdot y
&\frac01&\frac11&\frac21&\frac31&\frac41&\frac51
&\frac02&\frac12&\frac22&\frac03&\frac13
\\
\hline
\frac01
&\frac01&\frac01&\frac01&\frac01&\frac01&\frac01
&\frac02&\frac02&\frac02
&\frac03&\frac03
\\
\frac11
&\frac01&\frac11&\frac21&\frac31&\frac41&\frac51
&\frac02&\frac12&\frac22&\frac03&\frac13
\\
\frac21
&\frac01&\frac21&\frac41&\frac01&\frac21&\frac41
&\frac02&\frac22&\frac12
&\frac03&\frac03
\\
\frac31
&\frac01&\frac31&\frac01&\frac31&\frac01&\frac31
&\frac02&\frac02&\frac02
&\frac03&\frac13
\\
\frac41
&\frac01&\frac41&\frac21&\frac01&\frac41&\frac21
&\frac02&\frac12&\frac22
&\frac03&\frac03
\\
\frac51
&\frac01&\frac51&\frac41&\frac31&\frac21&\frac11
&\frac02&\frac22&\frac12
&\frac03&\frac13
\\
\hline
\frac02
&\frac02&\frac02&\frac02&\frac02&\frac02&\frac02
&\frac02&\frac02&\frac02
&\bot&\bot
\\
\frac12
&\frac02&\frac12&\frac22&\frac02&\frac12&\frac22
&\frac02&\frac22&\frac12
&\bot&\bot
\\
\frac22
&\frac02&\frac22&\frac12&\frac02&\frac22&\frac12
&\frac02&\frac12&\frac22
&\bot&\bot
\\
\hline
\frac03
&\frac03&\frac03&\frac03&\frac03&\frac03&\frac03
&\bot&\bot&\bot
&\frac03&\frac03
\\
\frac13
&\frac03&\frac13&\frac03&\frac13&\frac03&\frac13
&\bot&\bot&\bot
&\frac03&\frac13
\\[3mm]
\end{array}
$
\caption{Multiplication of fracpairs in $\CCcm(\Int/6\Int)$}
\label{tab:mul}
\end{table}

Separation of $1/1$ and $2/2$ in $\CCcm(\Int/6\Int)$
now follows easily: 
if $1/1=_\cceq2/2$ then 
\\
$\CCcm(\Int/6\Int)\models 1/2=2/2\cdot 2/1=1/1\cdot2/1$
and $\CCcm(\Int/6\Int)\models 1/2=2/2\cdot5/1=1/1\cdot5/1$, hence $2/1=_\cceq5/1$, which contradicts 
their separation mentioned in Appendix~\ref{app:01}. 

\newpage 

Addition in $\CCcm(\Int/6\Int)$ is defined in the following table, 
where we leave out $0/1$ (the zero for $+$) and
$1/0=\bot$ (recall $\Mda\vdash x+\bot=\bot$):
\[
\arraycolsep=12pt\def\arraystretch{2.5}
\begin{array}{c|ccccc|ccc|cc}
x+y
&\frac11&\frac21&\frac31&\frac41&\frac51
&\frac02&\frac12&\frac22&\frac03&\frac13
\\
\hline
\frac11
&\frac21&\frac31&\frac41&\frac51&\frac01
&\frac22&\frac02&\frac12
&\frac13&\frac03
\\
\frac21
&\frac31&\frac41&\frac51&\frac01&\frac11
&\frac12&\frac22&\frac02
&\frac03&\frac13
\\
\frac31
&\frac41&\frac51&\frac01&\frac11&\frac21
&\frac02&\frac12&\frac22
&\frac13&\frac03
\\
\frac41
&\frac51&\frac01&\frac11&\frac21&\frac31
&\frac22&\frac02&\frac12
&\frac03&\frac13
\\
\frac51
&\frac01&\frac11&\frac21&\frac31&\frac41
&\frac12&\frac22&\frac02
&\frac13&\frac03
\\
\hline
\frac02
&\frac22&\frac12&\frac02&\frac22&\frac12
&\frac02&\frac12&\frac22
&\bot&\bot
\\
\frac12
&\frac02&\frac22&\frac12&\frac02&\frac22
&\frac12&\frac22&\frac02
&\bot&\bot
\\
\frac22
&\frac12&\frac02&\frac22&\frac12&\frac02
&\frac22&\frac02&\frac12
&\bot&\bot
\\
\hline
\frac03
&\frac13&\frac03&\frac13&\frac03&\frac13
&\bot&\bot&\bot
&\frac03&\frac13
\\
\frac13
&\frac03&\frac13&\frac03&\frac13&\frac03
&\bot&\bot&\bot
&\frac13&\frac03
\end{array}
\]
Separation of $0/1$ and $0/2$ in $\CCcm(\Int/6\Int)$ can be shown using the table above: 
if $0/1=_\cceq 0/2$ then $\CCcm(\Int/6\Int)\models 1/1=0/1+ 1/1=0/2+1/1=2/2$, 
which contradicts their separation shown above.

\medskip

Finally, we provide a table for both additive and multiplicative inverse:
\[
\arraycolsep=10pt\def\arraystretch{2.5}
\begin{array}{l}
\begin{array}{c|c|cccccc|ccc|cc}
x
&\bot&\frac01&\frac11&\frac21&\frac31&\frac41&\frac51
&\frac02&\frac12&\frac22&\frac03&\frac13
\\
\hline
-x
&\bot&\frac01&\frac51&\frac41&\frac31&\frac21&\frac11
&\frac02&\frac22&\frac12
&\frac03&\frac13
\\[0mm]
x^{-1}
&\bot&\bot&\frac11&\frac12&\frac13&\frac22&\frac51
&\bot&\frac12&\frac22
&\bot&\frac13
\end{array}
\end{array}
\]
Note that in the particular case of $\CCcm(\Int/6\Int)$, the equation $(x^{-1})^{-1}=x^{-1}$ is valid. 

\newpage

\section{Some identities for common meadows}
\label{app:1}

\begin{enumerate}
\item $\Mda\vdash -(x\cdot y)=x\cdot (-y)$.

\begin{proof}
First, $\Mda\vdash 0\cdot x=0\cdot (-x)$ because $0\cdot x=x+-x=-x+-(-x)=0\cdot (-x)$. Hence
\begin{align*}
-(x\cdot y)&=-(x\cdot y)+0\cdot -(x\cdot y)
\\
&=-(x\cdot y)+0\cdot x\cdot y
\\
&=-(x\cdot y)+x\cdot(y+(-y))
\\
&=-(x\cdot y)+x\cdot y+x\cdot(-y)
\\
&=0\cdot(x\cdot y)+x\cdot(-y)
\\
&=(x+0\cdot x)\cdot(-y)
\\
&=x\cdot(-y).
\end{align*}
\end{proof}
\item
$\Mda\vdash x\cdot y^{-1}+u\cdot v^{-1}=(x\cdot v + u\cdot y)\cdot (y\cdot v)^{-1}$. 

\begin{proof}
First, $\Mda\vdash 0\cdot (x\cdot y)=0\cdot (x+y)$ because
\begin{align*}
0\cdot (x+y)&=0\cdot (x+y)\cdot(x+y)\\
&=0\cdot x+0\cdot x\cdot y +0\cdot y\cdot x+0\cdot y\\
&=(0+ 0\cdot y)\cdot x+(0+0\cdot x)\cdot y\\
&=0\cdot x\cdot y+0\cdot x\cdot y\\
&=0\cdot(x\cdot y),
\end{align*}
and thus
\begin{align}
\nonumber
x \cdot y \cdot y^{-1}
&=x\cdot(1+0\cdot y^{-1})
\\
\nonumber
&=x+0\cdot x\cdot y^{-1}\\
\nonumber
&=x+0\cdot x+0\cdot y^{-1}
\\
\label{eq:cv}
&=x+0\cdot y^{-1}.
\end{align}
Hence,
\begin{align*}
(x \cdot v + u \cdot y) \cdot (y \cdot v)^{-1}&=
x \cdot y^{-1} \cdot v \cdot v^{-1}+
u \cdot v^{-1} \cdot y \cdot y^{-1}\\
&=(x \cdot y^{-1}+0\cdot v^{-1})+
(u \cdot v^{-1} +0 \cdot y^{-1})
&&\text{by \eqref{eq:cv}}\\
&=(x \cdot y^{-1}+0\cdot y^{-1})+
(u \cdot v^{-1} +0 \cdot v^{-1})\\
&=x \cdot y^{-1}+u \cdot v^{-1}.
\end{align*}
\end{proof}
\item
$\Mda\vdash x\cdot (x^{-1}\cdot x^{-1})=x^{-1}$.

\begin{proof}
$x\cdot x^{-1}\cdot x^{-1}=
(1+0\cdot x^{-1})\cdot x^{-1}
=x^{-1}+0\cdot x^{-1}
=x^{-1}$.
\end{proof}
\end{enumerate}

\begin{thebibliography}{58}

\bibitem{BP85}
Bachmair, L. and Plaisted,  D.A. (1985).
Termination orderings for associative commutative rewriting systems.
\newblock \emph{Journal of Symbolic Computation}, (1):329--349.

\bibitem{BB15}
Bergstra, J.A. and Bethke, I. (2015).
Subvarieties of the variety of meadows.
\newblock \texttt{arXiv:1510.04021v2} [math.RA], 22 October 2015.

\bibitem{BBP13}
Bergstra, J.A., Bethke, I., and Ponse, A. (2013).
Cancellation meadows: a generic basis theorem and some applications.
\newblock \emph{The Computer Journal}, 56(1):3--14, 
\texttt{doi:10.1093/comjnl/bsx147}.

\bibitem{BM11a}
Bergstra, J.A. and Middelburg, C.A. (2011).
Inversive meadows and divisive meadows.
\newblock \emph{Journal of Applied Logic}, 9(3):203--220.
DOI: 10.1016/j.jal/2011.03.001. 
Also available as \texttt{arXiv:0907.0540v3} [math.RA], 2 November 2010.

\bibitem{BP14}
Bergstra, J.A. and Ponse, A. (2015).
Division by zero in common meadows.
\newblock In R. de Nicola and R. Hennicker (editors), 
\emph{Software, Services, and Systems} (Wirsing Festschrift), 
LNCS 8950, pages 46--61, Springer.
\newblock Also available as \texttt{arXiv:1406.6878v2} [math.RA], Dec 2014, and
Report TCS1410v2, University of Amsterdam, 
\url{https://ivi.fnwi.uva.nl/tcs/publications.php#reports}.

\bibitem{BP14a}
Bergstra, J.A. and Ponse, A. (2014).
Three datatype defining rewrite systems for datatypes of Integers each extending 
a datatype of Naturals.
\newblock \texttt{arXiv:1406.3280v2} [math.RA], 21 August 2014.
Also available as Report TCS1409v2, University of Amsterdam, 
\url{https://ivi.fnwi.uva.nl/tcs/publications.php#reports}.

\bibitem{BT95}
Bergstra, J.A. and Tucker, J.V. (1995).
Equational specifications, complete term rewriting systems, and 
computable and semicomputable algebras. 
\newblock \emph{Journal of the ACM}, 42(6):1194--1230.

\bibitem{BT07}
Bergstra, J.A. and Tucker, J.V. (2007).
The rational numbers as an abstract data type.
\newblock {\em Journal of the ACM}, 54(2), Article 7.

\bibitem{BR10}
Bethke, I and  Rodenburg, P.H. (2010). 
The initial meadows. 
\newblock \emph{The Journal of Symbolic Logic}, (75):888--895. 
DOI: 10.2178/jsl/1278682205. 

\bibitem{Bourbaki}
Bourbaki, N.  
\footnote{Bourbaki group, officially known as the 
 \emph{Association des collaborateurs de Nicolas Bourbaki.}}
(1990).
\newblock \emph{Algebra II}, Chapter V, \S 7. Series: Elements of Mathematics,
Springer-Verlag.
(English translation of: \emph{Alg\`ebre}, chapitres 4 \`a 7. Masson, Paris 1981.)

\bibitem{Contejean}
Contejean, E., March\'e, C., and	Rabehasaina, L. (1997).
Rewrite systems for natural, integral, and rational arithmetic.
In: Hubert Comon (ed.),
\emph{RTA '97, Proceedings of the 8th International Conference on Rewriting 
Techniques and Applications}, pages 98--112, Springer-Verlag.

\bibitem{Griesel}
Griesel, H. (1981).
Der quasikardinale Aspekt in der Bruchrechnung.
\newblock \emph{Der Mathematikunterricht}, 27(4):87--95.

\bibitem{Lam}
Lam, T.Y. (2001).
\emph{A First Course in Noncommutative Rings} (Second Edition).
\newblock Series: Graduate Texts in Mathematics. Springer.

\bibitem{Matsumura}
Matsumura, H. (translated by M. Reid) (1989).
\emph{Commutative Ring Theory}.
Cambridge University Press.

\bibitem{Padberg}
Padberg, F. (2012).
\emph{Didaktik der Bruchrechnung} (Fourth Edition).
\newblock Series: Mathematik, Primar- und Sekundarstufe, Springer-Spektrum.

\bibitem{Rollnik}
Rollnik, S. (2009).
Das pragmatische Konzept f\"ur den Bruchrechenunterricht. PhD
thesis, University of Flensburg, Germany.

\end{thebibliography}
\end{document}